 \documentclass{ajour}
\usepackage{epsf}
\usepackage{latexsym}
\usepackage{amsfonts}
\usepackage{amssymb}

\newcommand{\one}{{\bf 1}}

\begin{document}

%


 \authorrunninghead{M.C. de Bondt}
 \titlerunninghead{An ode to Phipps' jeep convoys}





\title{An ode to Phipps' jeep convoys}

\author{Michiel de Bondt}
\affil{Department of mathematics, University of Nijmegen, \\
       Toernooiveld 1, 6525 ED Nijmegen, The Netherlands}

\email{MichieldeB@netscape.net}

%







\abstract{The jeep problem was first solved by O. Helmer
and N.J. Fine.
But not much later, C.G. Phipps formulated a more general
solution. He formulated a so-called convoy or caravan variant
of the jeep problem and reduced the original problem to it.

We shall refine the convoy idea of Phipps and subsequently 
view a more general jeep problem, which we solve for 
jeep convoys as well as for a single jeep. In the last section 
we solve Maddex' jeep problem.}


\begin{article}


\section{INTRODUCTION} \label{intro}

The original jeep problem is formulated as follows. Given
a jeep that can carry one tankload of fuel and can travel one
distance unit per tankload. The jeep is required to cross a desert
$d$ units wide. To do so, it may make depots of fuel in the 
desert. How much fuel is required at the border of the
desert.

In case $x = 1 \frac12$, $2 \frac56$ tankloads suffice. 
The jeep can do the following steps.
\begin{enumerate}

\item Ride to $\frac16$ with a full fuel tank. Dump $\frac23$ 
      tankload at $\frac16$ and ride back to $0$, where you arrive
      empty.

\item Ride to $\frac16$ with a full fuel tank. Take $\frac16$
      tankload from the depot at $\frac16$. Ride farther to $\frac12$.
      Dump $\frac13$ tankloads of fuel at $\frac12$ and ride back 
      to $\frac16$. Take $\frac16$ tankload form the depot at 
      $\frac16$ and ride back to $0$, where you arrive empty.

\item Ride to $\frac16$ with the remaining $\frac56$ tankloads of fuel
      at position $0$. Take the remaining $\frac13$ tankload of fuel
      from position $\frac16$ and ride farther to position $\frac12$.
      Take the $\frac13$ tankload of fuel from this position and ride 
      to $1 \frac12$.

\end{enumerate}

\begin{figure}[ht] \label{jeepexamp}

\begin{center}
\begin{picture}(260,335)(0,-7)

\put(-30,40){\epsfbox{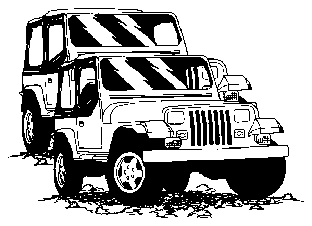}}
\put(30,95){\epsfbox{jeeps.ps}}
\put(210,275){\epsfbox{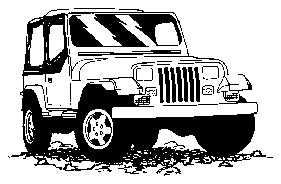}}

\put(-15,0){\makebox(0,0){$2\frac56$}}
\put(-22,-7){\dashbox(14,14)}
\put(-8,0){\dashbox(16,0)}
\put(0,0){\makebox(0,0){$\bullet$}}
\put(15,0){\makebox(0,0){$1\frac56$}}
\put(8,-7){\dashbox(14,14)}
\put(2,4){\vector(1,2){11}}
\put(15,30){\makebox(0,0){$\bullet$}}
\put(15,30){\dashbox(8,0)}
\put(30,30){\makebox(0,0){$\frac23$}}
\put(23,23){\dashbox(14,14)}
\put(17,26){\vector(1,-2){11}}
\put(22,0){\dashbox(46,0)}
\put(30,0){\makebox(0,0){$\bullet$}}
\put(75,0){\makebox(0,0){$\frac56$}}
\put(68,-7){\dashbox(14,14)}
\put(82,0){\dashbox(38,0)}
\put(32,4){\vector(1,2){11}}
\put(45,30){\makebox(0,0){$\bullet$}}
\put(37,30){\dashbox(31,0)}
\put(75,30){\makebox(0,0){$\frac12$}}
\put(68,23){\dashbox(14,14)}
\put(47,34){\vector(1,2){26}}
\put(75,90){\makebox(0,0){$\bullet$}}
\put(75,90){\dashbox(38,0)}
\put(120,90){\makebox(0,0){$\frac13$}}
\put(113,83){\dashbox(14,14)}
\put(77,86){\vector(1,-2){26}}
\put(105,30){\makebox(0,0){$\bullet$}}
\put(82,30){\dashbox(31,0)}
\put(120,30){\makebox(0,0){$\frac13$}}
\put(113,23){\dashbox(14,14)}
\put(107,26){\vector(1,-2){11}}
\put(120,0){\makebox(0,0){$\bullet$}}
\put(122,4){\vector(1,2){11}}
\put(135,30){\makebox(0,0){$\bullet$}}
\put(127,30){\dashbox(8,0)}
\put(137,34){\vector(1,2){26}}
\put(165,90){\makebox(0,0){$\bullet$}}
\put(127,90){\dashbox(38,0)}
\put(167,94){\vector(1,2){86}}
\put(255,270){\makebox(0,0){$\bullet$}}

\end{picture}
\end{center}

\caption{An illustration of an algorithm to cross a desert of
$1\frac12$ units.} 

\end{figure}

The above algorithm is illustrated in figure \ref{jeepexamp}.

Both round trips are indicated by a double jeep (steps 1.\@ and 
2.) and a single jeep indicates the outward trip (step 3.). Single 
and double jeeps play an important role in the next sections.

Dealing with jeep problems, I decided that it is a better idea
to let the jeeps ride from right to left, since solving jeep
problems often requires to view things reversely. As a compromise,
the jeep in figure \ref{jeepexamp} rides from below to above.
The reader can turn the manuscript a quarter in the desired direction.
The amount of fuel of depots is given in dashed boxes.

The jeep problem is solved in \cite{fine} and \cite{dozen}.
Additionally, O. Helmer solved the jeep problem in a series of two
papers according to M. Pollack in \cite{delivery}.
In \cite{crossing}, the problem of crossing a desert of
2 units is considered, but the solution is
essentially that of the general problem. 

A variant of the
jeep problem that was considered by O. Helmer as well 
is the so called round trip jeep problem,
where the jeep has to return to the border of the desert
after crossing the desert. An obtainable solution can be found in 
\cite{dozen}, except that some flaws are corrected one year later.

The above algorithm is {\em normal} in the sense that no
abstract formulations are used and consists of two round trips 
(1.\@ and 2.) 
and one outward trip (3.) into the desert. Instead of doing
these trips after each other, one can also do them at the same 
time, provided 3 jeeps are available. This way we get a convoy
formulation of the problem. C.G. Phipps reasoned that
a round trip jeep can be seen as an outward trip jeep that
consumes twice as much fuel, see \cite{phipps} and \cite{return}. 
We will discuss this view in section \ref{convoy}. 

D. Gale solved the jeep variant with more jeeps involved, but
did not use the convoy formulation of Phipps, since the argument 
that any jeep algorithm can be seen as a convoy algorithm with all
jeeps traveling together seems to be quite incomplete, see 
\cite{dozen}. That is why in \cite{roundtrip}, the authors refer
to \cite{dozen} for the solution of the round trip jeep problem.
Since Phipps' ideas were brilliant, he deserves much better. That is
why this article is titled as it is.

\section{PHIPPS' CONVOY OF JEEPS} \label{convoy}

Consider the normal algorithm for crossing a desert of $1 \frac12$
units in section \ref{intro}. If it is done with three jeeps
such that all riding jeeps are at the same position all the time,
we get a {\em normal convoy algorithm}.

We get the {\em forward convoy formulation} if we add the 
return trips of the round trip jeeps to the outward trips.
In this formulation, round trip jeeps use twice as much fuel 
per unit as outward trip jeeps, since they pass each position
between $0$ and their farthest point from both directions.

So we can see the round trip of a normal jeep as a single 
trip with a {\em double jeep}: a jeep that can carry two tankloads
of fuel and uses two tankloads per unit. 
The amount of fuel at some moment $t$ in a double jeep is
$o + (1 - r)$, where $o$ is the amount of fuel in the corresponding
moment in the outward part of the round trip and $r$ is the amount
of fuel in the corresponding moment in the return part of the outer 
trip. At the farthest position of a round trip, $r$ equals $o$,
so the double jeep finally keeps one tankload of fuel in its tank.

We get the following {\em forward convoy formulation} of the algorithm 
in section \ref{intro}.
\begin{enumerate}

\item Ride to $\frac16$ with two double jeeps and a single jeep,
      taking $4 \frac56$ tankloads of fuel from position $0$. 
      At $\frac16$, 4 tankloads of fuel remain. Transfer fuel 
      of one double jeep to the other jeeps, such that the other 
      jeeps get full fuel tanks.

\item Ride to $\frac12$ with the full double jeep and the single jeep,
      leaving the other double jeep with 1 tankload at $\frac16$. 
      At $\frac12$, 2 tankloads of fuel remain. Transfer fuel of the 
      remaining double jeep to the single jeep, such that the single 
      jeep gets a full tank.

\item Ride to $1 \frac12$ with the single jeep, leaving the remaining
      double jeep with 1 tankload at $\frac12$.

\end{enumerate}
Notice that the jeeps of the above forward convoy formulation 
correspond to the steps of the original formulation in section 
\ref{intro}.
In the forward convoy formulation, 2 tankloads more are used, 
but these tankloads are still in the double jeeps, since they
finally keep one tankload of fuel each. So in fact, double 
jeeps that are used start with one tankload of additional fuel,
which has to be restored finally.

It is not important where the fuel is at some moment
that the convoy is progressing into the desert. Since the jeeps
that are still in the convoy are all together, fuel can be 
exchanged as soon as one jeep gets empty.
But for making a jeep algorithm without convoy from the above 
convoy algorithm, it is crucial that the jeeps are ordered, 
with the single jeep highest in order, such that fuel is only be 
transferred from lower to higher jeeps. 

In the above algorithm, the double jeep 
first left is the lowest jeep. The order of jeeps corresponds
to the order in time of the trips from $0$ from the original 
jeep without convoy in section \ref{intro}. If the jeeps can not 
be ordered as above, then in the corresponding normal algorithm
for one jeep, the jeep must use fuel 
from a depot that is not carried yet to the depot, which is 
impossible.

If fuel is transferred to a double jeep, half of it is used in 
the outward part and the other half in the return part of the round
trip in the corresponding formulation without a convoy.

Instead of questioning how much fuel is required to cross a desert,
we can also question how far the jeep can get with some amount
of fuel. Both problems are essentially the same.

The above forward convoy formulation can easily be generalized 
to arbitrary amounts of initial fuel at position $0$. We get the 
following general convoy formulation.
\begin{enumerate}

\item Take one single jeep and $n_2 := \lceil x \rceil - 1$ double jeeps, where
      $x$ is the initial amount of fuel at $0$. Fill all jeeps
      completely with fuel, except the double jeep that is lowest 
      in order: fill that jeep with $x - \lceil x \rceil + 2$
      tankloads of fuel. The total amount of fuel is now $1 \cdot 1 + 
      (\lceil x \rceil - 2) \cdot 2 + 1 \cdot (x - \lceil x \rceil + 2) =
      x + (\lceil x \rceil - 1) = x + n_2$.

\item If the single jeep is the only jeep that is remained in the
      convoy, then ride into the desert with that jeep until there
      is no fuel left. Otherwise, ride into the desert with the 
      convoy until the total amount of fuel of the convoy becomes 
      an even integer. 
      Transfer all fuel above one tankload of the lowest jeep that 
      is still in the convoy to the other jeeps, which will get 
      completely filled. Remove the lowest jeep from the convoy 
      and repeat this step.

\end{enumerate}
At the end, there are $n_2$ tankloads of fuel left: one tankload in each 
double jeep. So $x$ tankloads of fuel are used.

Instead of using the forward convoy formulations as suggested 
by C.G. Phipps, we will use {\em backward convoy formulations}
in the remaining of this article. The convoy is now
riding from position $d$ at the other side of the desert to 
position $0$. The convoy starts with one single jeep, and riding
back to $0$, double jeeps are added to the convoy. The amount
of fuel $f$ in some jeep in the backward convoy formulation 
at some moment corresponds to the amount of emptiness of
the same jeep at the corresponding moment in the forward
convoy formulation (i.e.\@ $1-f$ tankloads of fuel for a single 
jeep and $2-f$ tankloads for a double jeep).

The backward convoy formulation with $d = 1 \frac12$ is as follows:
\begin{enumerate}

\item Create a full single jeep at $1 \frac 12$ and ride 
      to $\frac12$. The single jeep is now empty.

\item Create a double jeep at $\frac12$ with one tankload of fuel. 
      Transfer $\frac13$ tankload of this fuel to the single 
      jeep. Ride with both jeeps to $\frac16$. Both jeeps are 
      now empty.

\item Create a second double jeep with one tankload of fuel. Transfer
      $\frac13$ tankload of fuel to the other double jeep and another
      $\frac13$ tankload of fuel to the single jeep. Ride with all
      jeeps to $0$.

\end{enumerate}
We see that in the backward convoy formulation, fuel is transferred
from lower to higher jeeps as well. This is because both time and
amount of fuel are inverted.

Further, we see that in step 2.\@ the relative amount of fuel 
in both jeeps is the same after the transfer (i.e.\@ a third).
Therefore, both jeeps are empty after the same amount of units
from the transfer. However, in step 3.\@, the absolute amounts
of fuel becomes the same due to the transfers, but not the relative
amounts of fuel. The single jeep gets relatively twice as much
as the other jeeps. For that reason, the single jeep arrives
at $0$ with $\frac16$ tankload of fuel, while both double jeeps
arrive empty. Each jeep starts with one tankload of fuel
at creation, so $3 - \frac16 = 2 \frac56$ tankloads are needed.

We generalize the backward convoy formulation now.
\begin{enumerate}

\item Create one single jeep at $d$ and ride to $0$.

\item If the convoy gets out of fuel, create a double jeep with
      one tankload of fuel. Distribute the tankload of fuel such that
      each jeep of the convoy gets the same relative amount of 
      fuel. Advance to $0$ with the convoy and repeat this step.

\end{enumerate}

We see that the backward convoy formulation is quite short. 
Further, we see that 2.\@ of the above is in fact an event-handler
of the algorithm. We shall formulate forthcoming algorithms by way of
event-handlers, since it seems more natural to do so.

An algorithm that is similar to a forward or backward
convoy formulation is called a {\em forward} or {\em backward convoy 
algorithm} respectively.
So each forward or backward convoy formulation is a forward 
or backward convoy algorithm respectively, but a forward or backward 
convoy algorithm does not need to be derived from a normal (convoy) 
algorithm.

We call an algorithm in which one or more jeeps are involved
a {\em Phipps algorithm} or an algorithm of Phipps' type, 
if each jeep only changes direction from backward to 
forward at position $0$.

\begin{proposition} \label{convoyform}
A backward convoy algorithm can be formulated as a normal convoy
algorithm of Phipps' type and vice versa.
\end{proposition}

\begin{proof}
In this section, we reformulated a single jeep algorithm of
Phipps' type as a backward convoy algorithm. More generally,
we can reformulate any Phipps algorithm as such. This completes
the vice versa part of this proposition.

Suppose we have a backward convoy algorithm. If we wish to 
reformulate it as a normal convoy algorithm, then we must specify
how double jeeps are split in outward trips and return trips. 
Further, we must ensure that fuel can be exchanged between jeeps. 
This must be done by way of depots on the ground.

First, we formulate a normal convoy algorithm where jeeps
may `borrow' fuel from the ground. So some positions might contain 
a negative amount of fuel temporarily. After that, we remove the
`borrowing'.

We see a double jeep with $2r$ tankloads of fuel in a backward
convoy algorithm as a return trip with $r$ tankloads of fuel in the
jeep's tank and an outward trip with $1-r$ tankloads of fuel in the
jeep's tank. 

We organize the normal convoy algorithm with `borrowing' as follows. 
First, all jeeps make their outward trip, each jeep leaving 
the convoy at the farthest position it reaches. After that, 
the farthest round trip jeep returns to $0$, taking the other 
round trip jeeps with it along the way.

Notice that `borrowing' might only be needed if a jeep dumps
fuel in the return trip (to neutralize a negative depot) or a jeep 
does not return empty at position $0$. The latter can be reduced to the 
former by demanding that jeeps that return at position $0$ dump all
their fuel after their return.

If fuel is dumped at position $x$ in a return trip, the 
dumped fuel comes from a farther position of the desert. 
But all fuel originally comes from the desert border, so
there is no need to dump fuel in the return trips. This
gives the desired result.
\end{proof}

\section{AN OPTIMAL ALGORITHM FOR PHIPPS' JEEP PROBLEM}

Phipps considered the problem of one jeep having to cross the
desert, supported by `helper jeeps'. The helper jeeps and possibly the
crossing jeep as well may be obliged to return to the desert border. We
consider the problem of $n$ jeeps having to cross the desert, of which $n_2$
jeeps do and $n_1 = n - n_2$ jeeps do not have to return to the desert border
eventually. The jeeps are supported by $m$ helper jeeps of which $m_1$ jeeps
do not need to return.

Phipps did not include depots to be filled in his
algorithm explicitly, but remarked that such additions can be made.
We formulate a so called extended backward convoy algorithm with such 
depots. In an {\em extended backward convoy algorithm}, we allow double jeeps 
with one tankload of fuel to disappear, in addition to the
rules of a backward convoy algorithm. Such a disappearance corresponds to a
change of direction from backward to forward at a position other than $0$, as
we will show in the next section. 
Single jeeps do not disappear in an extended convoy
algorithm, which is not very amazing since such a disappearance
would correspond to a jeep starting from that position instead of
position 0. This is impossible within the context of Phipps' jeep problem.

\begin{algo} \label{Finejeep}
Start with a convoy at position $d$ with $n$ jeeps initially, of which $n_1$
single jeeps and $n_2$ double jeeps, all with one tankload of fuel. Transfer fuel 
from the single jeeps to the double jeeps such that each jeep gets the same 
relative amount of fuel. If a depot has to be filled on position 
$d$, then call the handler of event 1 first. After that, ride 
to position $0$ with the whole convoy.

\begin{em}
Event 1: The convoy meets a position where a depot has to be filled.
\end{em} \\
{\em Handler:}
Do the handler of event 2 as many times as required in order to
get the amount of tank fuel larger than the amount of fuel the depot needs
(without advancing to the desert border).
After that, use fuel to fill the depot. Advance to the desert border, with
each jeep having the same relative amount of fuel.

\begin{em}
Event 2: The convoy runs out of fuel.
\end{em} \\
{\em Handler:} 
If the number of single jeeps is $n_1 + m_1$, then create a double jeep with
one tankload of fuel. Otherwise, create a single jeep with one 
tankload of tank fuel. Distribute the tankload of fuel among the convoy such that 
each jeep gets the same relative amount of fuel and advance to the desert 
border.
\end{algo}

In the rest of this section, we will prove that algorithm \ref{Finejeep}
is optimal. This seemed to be clear for Phipps, but D. Gale missed some 
arguments in Phipps' article. The largest problem is that jeeps may change 
direction from backward to forward at positions other than $0$. Below we show 
that such changes of direction correspond to cancelations of double jeeps.
 
\begin{proposition}
Any algorithm with one or more jeeps can be formulated as an
extended backward convoy algorithm.
\end{proposition}

\begin{proof}
Call a round trip starting at one position $p$ to a farther position $q$
a {\em forward loop (from $p$ to $q$)}. So exactly $2(q-p)$ units are 
traveled in a forward loop from $p$ to $q$. 

The reformulation is quite similar to that in proposition \ref{convoyform}
in section \ref{convoy}. Suppose that we have a 
normal algorithm with a forward loop. Take a maximal forward loop, i.e.\@
take the starting position of the forward loop minimal. Cancel the
maximal loop from the algorithm and replace it by a one way trip of
a double jeep from $q$ to $p$, starting with one tankload of fuel at $q$.
After the double jeep arrives at $p$, fuel is exchanged between the
double jeep and the jeep from which the forward loop is canceled, such 
that the double jeep gets an amount of one tankload of fuel in its tank.

To show that this method works, we must show that after exchanging
fuel between the double jeep and the jeep from which the forward loop is 
canceled, the amount of fuel of the latter jeep is the same as it 
was in the original algorithm.
Suppose that the jeep started its maximal forward loop with $o$
tankloads of fuel and ended it with $r$ tankloads of fuel. In order
to get the right amount of fuel, the jeep from which the forward loop
is canceled should get rid of $o - r$ tankloads of fuel if $o > r$
and similarly should get $r - o$ tankloads of fuel if $o < r$. Since
the amount of fuel of the double jeep is $1 - o + r$ just before
exchanging fuel with the jeep from which the forward loop is canceled,
the above is satisfied.

A problem of this construction is that fuel can not be exchanged,
in the sense that some positions might contain a negative amount of
fuel temporarily. This problem will disappear at the end of this
reformulation, since then we have an extended backward convoy algorithm 
where each position is passed only once by the convoy, and therefore 
temporary underflows of fuel at positions can not occur.

Remove maximal forward loops in the above way until they do not
exist any more. The remaining of the original algorithm is now
in fact of Phipps' type and can be transformed to a backward convoy
formulation as described in section \ref{convoy}. The above rides
of double jeeps can be inserted in the backward convoy, which yields
an extended convoy algorithm.
\end{proof}

\begin{theorem}
Algorithm \ref{Finejeep} is optimal.
\end{theorem}

\begin{proof}
Consider an optimal extended backward convoy algorithm $S$. We may assume that
in $S$, jeeps are only added to the backward convoy if the backward convoy 
runs out of fuel, since postponing adding a jeep can only save fuel.
Furthermore, single jeeps should be added first, since they use less fuel than 
double jeeps. This way we get algorithm \ref{Finejeep}.
\end{proof}

Now you read the above solution of the jeep problem, you could think 
the following: the backward convoy is funny, but the same results can
be established without it with half as many pages. This might be true, but 
I think that the backward convoy is a concept that is worth being 
displayed extensively, for both scientific and historical reasons. 

If there are depots with fuel to be used in the desert, then there are
several complication that might play up. If the amount of fuel of the convoy 
does not get larger than one tankload, then nothing serious happens. 
But if the convoy gets more than one tankload, then there is enough fuel to
cancel a double jeep, but this is only possible if there are more than $n_2$
double jeeps in the convoy. 

For that reason, it is no longer true that
single jeeps must be added first to the convoy and double jeeps
after that. It is neither the case that the convoy should start with $n_1$
single jeeps: it should start with $n$ jeeps of which at least $n_2$
double jeeps. With these adaptation, one can formulate an algorithm
where the types of jeeps are undetermined, and thus we get a {\em
nondeterministic algorithm}. But one can prove that such an algorithm is
optimal, which means that it is optimal for some way of chosing jeep types, as
long as the convoy is able to absorb all fuel to be used for all ways of 
chosing jeep types. 

However, if some depot contains more fuel than the backward convoy can 
absorb, then things get really harder. An idea is to create new jeeps in order to 
enlarge the fuel capacity of the backward convoy. But such a backward convoy 
algorithm can not be transformed to a normal algorithm in general. 
The problem is that the jeeps cannot take advantage of the depot fuel before
it is reached by some jeep. In a convoy formulation, time is in fact 
eliminated, whence this problem is not taken into account.

In \cite{jeap}, the case of a single jeep having to cross the desert with
arbitrary depots of fuel to be used and to be filled is solved.
The backward convoy is split in two parts there: one before and one after 
reaching the large depot of fuel. This problem is strongly related to Gale's 
round-trip problem, and generalizes Theorem 1 of \cite{roundtrip}. In the same 
article, another problem is formulated, which will be discussed in the next section.
This problem involves a so-called {\em Dewdney jeep}. Other problems with 
Dewdney jeeps are studied in \cite{jeap} as well. 

In \cite{roundtrip}, another jeep problem is formulated as well.
This jeep variant is also known as Klarner's camel-banana problem.
Now, the jeep is a camel that needs to eat a banana every unit and can carry 
one banana on its back.
It is solved in \cite{cb} and \cite{logistics}.
Only the outward case is considered in both references.

\section{MADDEX' JEEP PROBLEM}

Maddex' jeep problem is the following. Again, we have a jeep that must cross
the desert. This jeep has a fuel tank of one tankload. In addition, it can
carry $B$ cans of $C$ tankloads of fuel. It may make temporary depots, but such
depots must be made of cans and only can fuel may be used to fill them. So
the jeep's tank may be filled with fuel from a can, but not vice versa.

At the desert border, there is an unlimited amount of fuel, but there is only
a finite number, say $N$, of cans. These cans can be filled at the desert
border. We solve both the outward trip case and the round trip case of 
Maddex' problem, so there is only one jeep, a so-called {\em Dewdney jeep},
and there are no additional depots to be made or used in the desert.

Maddex' round trip jeep problem is not very hard. Let $D_N = 
(N \cdot C + 1)/2$. After leaving the fuel station, the jeep
can ride only $N \cdot C + 1$ units before it must return to
the fuel station, since the total capacity of all cans and the
jeep's own fuel tank together is $N \cdot C + 1$. So the jeep
cannot get farther than $D_N$. The following proposition shows 
that this upper bound can be achieved.

\begin{proposition}
With $N$ cans, the jeep can dump a full can at $D_{N-1}$, without
using fuel of this can, and eventually return to the fuel station, 
without leaving one of the other cans somewhere in the desert.
\end{proposition}

\begin{proof}
We show the case $C \le 1$ first.
By induction, it follows that the jeep can dump $N-1$ cans at 
$D_{N-2}, D_{N-3}, \ldots, D_1, D_0$, in this order, 
without using the last can. After doing this, the jeep rides
to $D_{N-1}$ with the last can filled, using fuel of the other cans.
At last, cans are retrieved in the order $D_0, D_1, \ldots, D_{N-3}, 
D_{N-2}$, which can be done in the same way as dumping full cans
at these positions. 

In case $C > 1$, the jeep cannot take a whole can in its tank, whence
the above is not possible. But the proposition is still valid, since the 
jeep can take the can from which he takes fuel along with it, except in case
it carries the can that is meant for $D_{n-1}$. But if it carries that can
half a distance unit at a time, then it can use tank fuel to do so.

Inductively dump two cans at $D_{N-2}$, and the other 
cans at $D_{N-3}, \ldots, \allowbreak D_1, \allowbreak D_0$. 
cans at $D_{N-2}$ half a unit farther with fuel of the other cans and 
restore the used fuel. Carry the can at $D_{N-2} + \frac12$ to $D_{N-1}$ or
another half a distance unit farther with the fuel of the other cans, etc.
\end{proof}

If $C > 1$ and $B \ge 2$, then cans do not need to be carried farther in
steps of half a unit, which the reader may show. We formulate a 
pseudo-algorithm for the outward case of Maddex' jeep problem.

\begin{algo} \label{Maddex}
Dump full cans at $D_{N-1}, D_{N-2}, \ldots, D_1, D_0$, eventually
returning at the fuel station. Start from the fuel station with
one tankload of tank fuel. Each time you meet a can, take the
fuel of it and advance. Advance as a single jeep if the total
amount of fuel of the jeep is $B \cdot C + 1$ at most. Otherwise,
advance as a triple jeep until the amount of fuel becomes
$B \cdot C + 1$ (which will be before the next can) and then advance
as a single jeep again.
\end{algo}

Now we have seen double jeeps, the triple jeeps in algorithm \ref{Maddex}
should not be a problem. The remainder of this section is devoted to show 
the optimality of algorithm \ref{Maddex}.

Let $t$ be the last moment that the jeep is at the fuel station.
Number the cans $1, 2, \ldots, N$ and let $x_i$ be the position
of can $i$ at moment $t$. Without loss of generality, we assume
that $x_N \le x_{N-1} \le \cdots \le x_2 \le x_1$. Let $c_i$
be the amount of fuel of can $i$ at moment $t$ for all $i$.

\begin{lemma} \label{Maddexup}
If $1 \le k \le N$, then
$$
2 \sum_{i=1}^k \max\{x_i - D_{N-i}, 0\} \le \sum_{i=1}^k (C - c_i)
$$
\end{lemma}

\begin{proof}
Notice that at least $k$ times before moment $t$, the jeep must transport a can to $x_k$ 
that will not return to a smaller position any more. Assume that for such a moment, 
there is another can that is transported to $x_k$ later, after which 
it is used to partially refill the first can. If the second can does return 
to a smaller position than $x_k$ after this refill, then we can interchange
the roles of both cans during the refill. Therefore, we may assume that the second can
does not return to a smaller position than $x_k$ any more after the refill either.

When the jeep transports a can to $x_k$ that will not return
to a smaller position any more for the $i^{\rm th}$ time, it can subsequently
get as far as $D_{N-i}$ by using fuel other than that from the $i$ cans 
that will stay farther than $x_k$, but in order to reach $x_i$,
$2 \max\{x_i - D_{N-i}, 0\}$ additional tankloads of fuel are 
necessary. This fuel cannot be restored any more due to the above assumption. 
Taking the sum from $1$ to $k$ gives the desired result.
\end{proof}

After moment $t$, it is clear that the Dewdney jeep should ride into the desert
and absorb all fuel it encounters, becoming a triple jeep just as in
algorithm \ref{Maddex} when necessary. But the positions and amounts
of fuel of the cans may be different. It is however equally expensive 
to ride as a triple jeep instead of a single jeep from $D_{n-i}$ to $x_i$ as
to transport a can from $D_{n-i}$ to $x_i$ before moment $t$.
So we get the following result.

\begin{theorem}
The solution of Maddex' jeep problem with $N$ cans is $D_N$ in case of a
round trip and the distance the 
Dewdney jeep reaches in algorithm \ref{Maddex} in case of an outward trip.
\end{theorem}

If more jeeps need to make a round trip, then a distance larger than
$D_N$ can be crossed, even if there must be a moment that all jeeps are at 
the farthest position simultaneously. The reader may show this. 









\end{article}
\end{document}